\documentclass[9pt,reqno]{amsart}
\usepackage{fullpage}
\usepackage[T1]{fontenc}                                   
\usepackage{amsfonts}
\usepackage[utf8]{inputenc}                               
\usepackage{comment}                                       
\usepackage{mparhack}                                      
\usepackage{amsmath,amssymb,amsthm,mathrsfs,eucal}                      
\usepackage{booktabs}                                                                          
\usepackage{graphicx,subfig}                                                                
\usepackage{wrapfig}                                                                            
\usepackage[bookmarks=true,colorlinks=true]{hyperref}                      
\usepackage{bm}                                                                                   

\newtheorem{defn}{Definition}[section]
\newtheorem{thm}{Theorem}[section]

\newtheorem{lem}{Lemma}[section]

\newcommand{\R}{\mathbb{R}}

\renewcommand{\d}{{\rm d}}

\newcommand{\sfd}{{\sf d}}
\newcommand{\X}{{\sf X}}
\newcommand{\Y}{C}
\newcommand{\lims}{\varlimsup}
\newcommand{\f}{{\sf pen}}
\newcommand{\Lip}{{\rm Lip}}

\def\lipa#1#2{{\rm lip}_a (#1,#2)}
\def\ep{\varepsilon}
\newcommand{\Ch}{{\rm Ch}}

\newcommand{\restr}[1]{\lower3pt\hbox{$|_{#1}$}} 

\def\mm{\mathfrak m}

\begin{document}
\author{S. Di Marino, N. Gigli, A. Pratelli}
\address{Simone Di Marino, Universit\`a di Genova, Dipartimento di Matematica,  via Dodecaneso 35, 16146 Genova (GE), Italy}
\email{simone.dimarino@unige.it}
\address{Nicola Gigli, SISSA, Via Bonomea 265, 34136 Trieste (TS), Italy}
\email{ngigli@sissa.it}
\address{Aldo Pratelli, Universit\`a di Pisa, Dipartimento di Matematica, Largo Bruno Pontecorvo 5, 56127 Pisa (PI), Italy} 
\email{aldo.pratell@unipi.it}
\title{Global Lipschitz extension preserving local constants}
\date{}
\begin{abstract}
\noindent
The intent of this short note is to extend real valued Lipschitz functions on metric spaces, while locally preserving the asymptotic Lipschitz constant. We then apply this results to give a simple and direct proof of the fact that Sobolev spaces on metric measure spaces defined with a relaxation approach \`a la Cheeger are invariant under isomorphism class of mm-structures. 
\end{abstract}
\maketitle

\section{introduction}

Consider a metric space $(\X,\sfd)$, a closed subset $C\subset\X$ and a $L$-Lipschitz function $g:C\to\R$. A well-known result by McShane \cite{McShane} ensures that there is a $L$-Lipschitz extension of $g$ to the whole $\X$, i.e.\ a $L$-Lipschitz function $f:\X\to\R$ whose restriction to $C$ coincides with $g$. In fact, McShane proof comes with an explicit construction and a maximality argument: any such $f$ must lie between the  $L$-Lipschitz extensions  $f^\pm$ of $g$ defined by
\[
\begin{split}
f^+(y)&:=\inf_{x\in C}g(x)+L\sfd(x,y),\\
f^-(y)&:=\sup_{x\in C}g(x)-L\sfd(x,y).
\end{split}
\]
For other results in this direction see also Milman's extension theorem in \cite{Milman98}.

\bigskip

In this paper we shall also consider a Lipschitz extension problem, but our goal is to preserve  not only  the global Lipschitz constant, but also the \emph{asymptotic Lipschitz constant}. Let us recall the definition of such quantity. For $g:C\to\R$ and $A\subset C$ the Lipschitz constant of $g$ on $A$ is defined as
\[
{\rm Lip} (g,A): = \sup \Big\{ \frac {|g(y_1)-g(y_2)|}{\sfd(y_1,y_2)} \; : \;  y_1,y_2 \in A\: , \;y_1 \neq y_2 \Big\}.
\]
Then for every $x\in C$ the asymptotic Lipschitz constant of $g$ at $x$ is given by
\[
\lipa{g}{x}:=\inf_{r > 0} {\rm Lip} ( g, C\cap B_r(y))=\lim_{r \to 0} {\rm Lip} ( g, C\cap B_r(y)).
\]
Observe that, for brevity, in the notation $\lipa{g}{x}$ we are omitting to emphasise the domain of definition of $g$, albeit this evidently has a role in the definition.

It is clear that if $f:\X\to\R$ is an extension of $g$, then the inequality $\lipa{g}{x}\leq \lipa{f}{x}$ holds for every $x\in C$.  Our goal is to find Lipschitz extensions that achieve the equality.

\begin{thm}\label{thm:main} Let $(\X,\sfd)$  be a metric space, $C\subset \X$ a subset and  $g: C \to \mathbb{R}$ a $L$-Lipschitz function. Then for every $\ep>0$ there exists an $(L+\ep)$-Lipschitz function $f: \X \to \mathbb{R}$ whose restriction to $C$ coincides with $g$ and such that
\begin{equation}
\label{eq:lipaug}
\lipa{f}{x}=\lipa{g}{x} \qquad \text{ for all  } x \in C.
\end{equation}
Moreover if $g$ is bounded (resp.\ with bounded support), then $f$ can be chosen to be bounded (resp.\ with bounded support).
\end{thm}
We collect some comments:
\begin{itemize}
\item[a)] Our result is sharp in the sense that we cannot maintain the same Lipschitz constant, even if we allow to increase the asymptotic Lipschitz constants by $\ep$. To see this consider the case $\X=[0,1]$, $C=\{0,1\}$ and $g(t):=t$ for $t\in C$. Then $g$ is 1-Lipschitz, $\lipa{g}{t}=0$ for any $t\in C$ and the only 1-Lipschitz extension of $g$ to the whole $\X$ is given by $f(t)=t$ for every $t\in\X$. However, for such $f$ we have $\lipa{f}{t}=1$ for every $t$.
\item[b)] Our proof is a kind of `localized' or `nonlinear' variant of McShane's argument, in particular, it is entirely constructive. This also means that we can build a right inverse of the restriction map $f\mapsto f\restr C$, for which the conclusion of our theorem hold, without using any form of Axiom of Choice.
\item[c)] It is unclear to us if one can extend functions while preserving the \emph{local Lipschitz constant} defined as ${\rm lip}(f,x):=\lims_{y\to x}\frac{|f(y)-f(x)|}{\sfd(x,y)}$.
\item[d)] Our result is relevant in, and motivated by, the study of Sobolev spaces over complete and possibly non-separable metric spaces equipped with a tight measure (that is, under the stated assumptions, a measure concentrated on a separable subset). Indeed, typically Sobolev spaces over metric measure spaces are studied in the separable case (see e.g.\ \cite{Cheeger00}, \cite{Shanmugalingam00}, \cite{HKST15}, \cite{AmbrosioGigliSavare11}, \cite{Bjorn-Bjorn11}) and one can wonder which of their properties remain valid in this slightly more general context. Moreover, the point of view generically adopted when studying lower Ricci curvature bounds is to consider a metric measure space $(\X,\sfd,\mathfrak m)$ to be isomorphic to $({\rm supp}(\mathfrak m),\sfd,\mathfrak m)$: this of course is possible only if all the relevant definitions are insensible to the existence of points outside the support of the measure. While this is clearly the case for the curvature-dimension condition (at least if one pays a bit of attention in stating it properly), things are more delicate for what concerns Sobolev spaces. There are indeed approaches to Sobolev functions, like the one based on the concept of modulus of a family of curves or the one based on the notion of test plan, for which clearly the existence of points outside the support of the measure is irrelevant; but for the definitions given via relaxation of some form of metric modulus of differential, like upper gradients, local Lipschitz constants or asymptotic Lipschitz constants, the situation is more complicated as these quantity \emph{are} affected by the behaviour of the function outside the support of the measure. 

There are various possible ways to see that these latter definitions of Sobolev spaces (in particular the one involving $\lipa fx$) are also invariant by isomorphism:
\begin{itemize}
\item[-] At least if $\X$ itself is separable, one can couple the  Lindelof property of $\X$ with a known property of Sobolev functions  (that is the locality of minimal weak upper gradients) to conclude. This is what has been done in \cite{DMGSP18}.
\item[-] One can check that all the arguments carried out in \cite{AmbrosioGigliSavare11},  \cite{AmbrosioGigliSavare11-3} that prove the equivalence of the various definitions of Sobolev functions do not really require separability of the space but only that of the support of the measure; then the conclusion would follow from the fact that the approach via test plan is invariant under isomorphism. Technically this is possible, and it works, but certainly it is a very indirect way to argue and seems an unnecessarily complicated argument.
\end{itemize}
None of these two approaches is really satisfactory. Instead, a direct consequence of our simple extension result is the invariance under isomorphism of metric measure structures of the definition of Sobolev functions via relaxation of the asymptotic Lipschitz constants: see Theorem \ref{thm:sobapp} and notice that its proofs does not require any knowledge of the structure of Sobolev functions or about other possible definitions.

We conclude emphasising that there are situations where it is natural to work with tight measures on non-separable spaces. In fact, this study is a byproduct of a research program devoted to the study of harmonic maps $u$ from ${\sf RCD}$ to ${\sf CAT}(0)$ spaces and in the process of doing so it is useful to consider the push-forward via $u$ of the measure on the source to endow the target metric space with a measure, and then to consider Sobolev functions in the resulting metric measure structure (see e.g.\ \cite{GPS18}, \cite{GT18}). Given that ${\sf CAT}(0)$ spaces are typically not separable but harmonic - and more generally metric-valued $L^2$ functions  - are defined to be essentially separably valued (for technical reasons analogue to those behind the same assumption when dealing with Bochner integration of Banach-valued maps), the need of studying Sobolev functions over non-separable spaces equipped with tight measures is explained.

\end{itemize}

\section{Proof of the main theorem}

We shall frequently make use of the following simple and  well known lemma:

\begin{lem}\label{lem:LLip} Let $(\X,\sfd)$ be a metric space, $A \subset \X$, $I$ a set of indexes and for every $i\in I$ let $f_i:\X\to\R$ be such that ${\rm Lip} (f_i,A)\leq L$. Then the function $f(x):= \inf_{i \in I} f_i(x)$ also satisfies ${\rm Lip} (f,A)\leq L$.
\end{lem}
\begin{proof} Let $x,y \in A$, $\ep>0$ and  $i\in I$ such that $f(x) \geq f_{i_x}(x)-\ep$. Using the assumption  that ${\rm Lip} (f_i,A)\leq L$ and the fact that  $f_{i} \geq f$ we get
\[
f(x) \geq f_{i}(x) - \ep \geq f_{i}(y) - L \, \sfd(x,y)  - \ep \geq f(y) - L \,\sfd(x,y) -\ep. 
\]
Reverting the roles of $x,y$ we deduce that  $|f(x)-f(y)| \leq L\, \sfd(x,y) + \ep$, so that the conclusion follows from the arbitrariness of  $\ep$.
\end{proof}

\begin{proof}[Proof of Theorem~\ref{thm:main}] It is not restrictive  to assume that $\ep\leq  L$ (as if $L=0$ the claim is trivial). Let us now consider a sequence $\{ \ep_k \}_{k \in \mathbb{Z}}$ such that:
\begin{itemize} 
\item[(i)] $\ep_k> 0$ for every $k\in\mathbb Z$; 
\item[(ii)] $k\mapsto\frac{\ep_{k-1}}{\ep_k} $ is increasing and goes to 0 when $k \to -\infty$;
\item[(iii)] for every $k \in \mathbb{Z}$ it holds
\begin{equation}\label{eqn:rappep}
\frac{ \ep_{k-1} }{\ep_k} \leq \frac \ep{3(L+\ep)}.
\end{equation}
\end{itemize} 
It is clear that such a sequence exists.

Then we will consider the  approximating slopes $S_k (x) := {\rm Lip} (g, C\cap B_{\ep_k} ( x))$ and the penalisation function  $ \f_x : [0,\infty) \to [0, \infty)$, defined as the only continuous function such that
\begin{equation}\label{eqn:deffx}
\f_x(0)=0 \qquad \qquad \f'_x(t) = S_k(x) + 3 L \frac{ \ep_{k-2}}{\ep_{k-1}} \qquad \text{ for }\ep_{k-2} < t < \ep_{k-1}.
\end{equation}
It is easy to see that  $(ii)$ grants that this is a good definition. Moreover, the fact that $k\mapsto S_k(x)$ is increasing and bounded by $L$ together with $(ii),(iii)$ ensure that
\begin{equation}
\label{eq:penconv}
\f'_x(t)\ \text{is bounded and increasing, i.e.\ $\f_x$ is convex and Lipschitz}.
\end{equation}
Then we put:
\begin{equation}\label{eqn:deff}
\begin{split}
\phi_x (y)&: = g(x) + \f_x ( \sfd(x,y))\qquad\forall x\in C,\ y\in \X \\
 f(y)&:= \inf_{x \in C} \left\{ \phi_x(y) \right\}\qquad\qquad\qquad\qquad\ \ \forall  y\in \X
\end{split}
\end{equation}
(Notice that the choice $\f_x(t) = L  t$ for every $t\geq 0$ would correspond to  McShane upper extension).
  
We will prove that $f$ is in fact the required extension for $g$ in several steps:\\

\noindent\textbf{Step 1.} We claim that
\begin{equation}
\label{eq:claim1}
\text{$\phi_x$ is $(L+\ep)$-Lipschitz for every $x\in C$.}
\end{equation}
By the very definition of $\phi_x$ it is sufficient to prove that   $\f_x$ is $(L+\ep)$-Lipschitz for any $x\in C$. To see this simply observe that for every $k\in \mathbb{Z}$ we have
\[
\f'_x(t) \stackrel{\eqref{eqn:deffx}}= S_k(x) + 3 L \frac{ \ep_{k-2}}{\ep_{k-1}} \stackrel{\eqref{eqn:rappep}}\leq L + 3 L  \frac {\ep}{ 3(L+\ep)} \leq L+ \ep, \qquad \forall t \in (\ep_{k-2}, \ep_{k-1})
\]
and  conclude by the arbitrariness  of $k \in \mathbb{Z}$.

\vspace{10pt}

\noindent\textbf{Step 2.} We claim that
\begin{equation}
\label{eq:claim2}
\text{whenever $x,y \in C$ and  $\sfd(x,y) \in [\ep_{k-1}, \ep_{k}]$, we have $\phi_x(y) \geq g(y) + \ep_{k-2} L$. }
\end{equation}
In fact, we have $g(x) \geq g(y)- S_k(x) \, \sfd(x,y)$ by the definition of $S_k(x)$, while  $\f_x(\sfd(x,y)) \geq \int_{\ep_{k-2}}^{\sfd(x,y)} \f'_x(t)\, \d t$ and thus:
\begin{align*}
\phi_x(y)& =g(x) + \f_x(\sfd(x,y)) \\
&\geq g(y)- S_k(x) \, \sfd(x,y)+ \int_{\ep_{k-2}}^{\sfd(x,y)} \f'_x(t)\, \d t  \\
 (\text{by }\eqref{eqn:deffx} \text{ and }\eqref{eq:penconv})\qquad\qquad&\geq  g(y) - S_k(x) \, \sfd(x,y) +  (\sfd(x,y) - \ep_{k-2}) \, \big(S_k(x) + 3 L \frac{ \ep_{k-2}}{\ep_{k-1}}\big) \\
 & = g(y) - \ep_{k-2} S_k(x) + 3L (\sfd(x,y) - \ep_{k-2}) \, \frac{ \ep_{k-2}}{\ep_{k-1}} \\
(\text{by }S_x(x)\leq L\text{ and }\sfd(x,y)\geq \ep_{k-1})\qquad\qquad & \geq  g(y) - \ep_{k-2} L + 3L (\ep_{k-1} - \ep_{k-2}) \, \frac{ \ep_{k-2}}{\ep_{k-1}}  \\
 &=  g(y) + \ep_{k-2} L + L (\ep_{k-1} - 3\ep_{k-2}) \, \frac{ \ep_{k-2}}{\ep_{k-1}}.
\end{align*}
To conclude notice that  \eqref{eqn:rappep} grants that $\ep_{k-1} \geq 3\ep_{k-2}$.

\vspace{10pt}

\noindent \textbf{Step 3.} We claim that 
\begin{equation}
\label{eq:claim3}
\text{ $f$ is an $(L+\ep)$-Lipschitz extension of $g$.}
\end{equation}
To this aim start noticing that  Step 2 ensures that  $\phi_x(y) \geq g(y)$ for every $x,y \in C$; by the very definition \eqref{eqn:deff} of $f$ this proves that $f(y) \geq g(y)$. On the other hand, trivially it holds $\phi_y(y)=g(y)$ for any $y\in C$, and thus  $f(y) \leq g(y)$. 

This shows that $f$ is an extension of $g$. The claim about the Lipschitz constant  follows directly from Lemma~\ref{lem:LLip} and Step 1.

\vspace{10pt}

\noindent \textbf{Step 4.}  We claim that
\begin{equation}
\label{eq:claim4}
\forall\bar x\in C,\ k\in\mathbb Z\quad\text{we have}\quad f(y)  = \inf_{x\in C \cap B_{\ep_{k}}(\bar{x}) }   \phi_x(y)  \qquad \forall y \in B_{\ep_{k-2}}(\bar{x}).
\end{equation}
In other words, for any   $\bar{x} \in C$, $k\in\mathbb Z$  and   $y \in B_{\ep_{k-2}}(\bar{x})$ the inf in \eqref{eqn:deff} does not change if we just consider $\phi_x$ with $x \in C \cap B_{\ep_{k}}(\bar{x})$.

To prove such claim we will show that for $\bar x,k,y$ as above and $x\in C$ with $\sfd(x,\bar x)\geq \ep_k$ we have 
\begin{equation}
\label{eq:claim44}
\phi_x(y)\geq f(y)+\ep_{k-1}\frac L3.
\end{equation}
Start noticing that by Step 3 we know that
\[
f(y) \leq g(\bar{x})+ \ep_{k-2} (L+\ep).
\]
On the other hand we have
$$\phi_x(y) \stackrel{\eqref{eq:claim1}}\geq \phi_x(\bar{x}) - \ep_{k-2}(L+\ep) \stackrel{\eqref{eq:claim2}}\geq g(\bar x) + \ep_{k-1} L - \ep_{k-2}(L+\ep).$$
The claim \eqref{eq:claim44} follows from these two inequalities, the bound  \eqref{eqn:rappep}  and the assumption $\ep\leq L$ made at the beginning of the proof.

\vspace{10pt}

\noindent \textbf{Step 5.} We claim that
\begin{equation}
\label{eq:claim5}
\text{For every $\bar x \in C$ we have $\lipa{f}{\bar x}=\lipa{g}{\bar x}$.  }
\end{equation}
Clearly it is sufficient to prove the inequality $\leq$. Let us fix $\bar x \in C$. By the definition of the asymptotic Lipschitz constant, it is sufficient to show that, for every $\bar{r}>0$ and $\xi>0$ we can find $r>0$ such that 
\begin{equation}
\label{eq:concl}
{\rm Lip} (f, B_r(\bar x)) \leq {\rm Lip} ( g, B_{\bar r}(\bar x)) + \xi.
\end{equation}

In order to prove this let us consider $k \in \mathbb{Z}$ such that $\ep_{k+3} < \bar r$ and $3L \frac {\ep_k}{\ep_{k+1}} < \xi $ (this is possible thanks to $(ii)$ which also ensures that $\ep_k\to0$ as $k\to-\infty$); then we claim that $r:=\ep_{k-2}$ will work.  To see this let $x\in C\cap B_{\ep_k}(\bar x)$ and notice that by \eqref{eqn:rappep} it easily follows that
\begin{subequations}
\begin{align}
\label{eq:p1}
B_{r}(\bar x)&\subset B_{\ep_{k+1}}(x)\\
\label{eq:p2}
 B_{\ep_{k+2}}(x)&\subset  B_{\ep_{k+3}}(\bar x)
\end{align}
\end{subequations}
and therefore
\[
{\rm Lip}(\phi_x,B_r (\bar{x})) \stackrel{\eqref{eq:p1}}\leq {\rm Lip}(\phi_x, B_{\ep_{k+1}} (x))\stackrel{\eqref{eqn:deffx}} \leq S_{k+2}(x) + 3L \frac {\ep_k}{\ep_{k+1}} \stackrel{\eqref{eq:p2}}\leq S_{k+3}(\bar{x}) +3L \frac {\ep_k}{\ep_{k+1}}. 
\]
Recalling our choice of $k$ we just proved that
\[
{\rm Lip}(\phi_x,B_r (\bar{x})) \leq{\rm Lip}(g,B_{\bar{r}} (\bar{x}))  +\xi\qquad\forall x\in C\cap B_{\ep_k}(\bar x)
\]
and the conclusion \eqref{eq:concl} follows by   Step 4 and  Lemma~\ref{lem:LLip}.

\vspace{10pt}

\noindent \textbf{Step 6.} We prove the last claims. If $g$ is bounded, then up to replacing $f$ with $-C\vee f\wedge C$ for $C>0$ sufficiently large we produce a bounded extension retaining all the required properties.

If moreover $g$ has bounded support (and thus in particular, being Lipschitz, is bounded), let $f$ be a bounded  extension satisfying \eqref{eq:lipaug} and $\Lip(f)\leq L+\ep/2$. Put $M:=\sup_{x\in\X}|f(x)|$ and let $\chi:\X\to[0,1]$ be a $\frac\ep {2M}$-Lipschitz function with bounded support which is identically 1 on a neighbourhood of $C$ (e.g.\ $\chi(x):=0\vee \big(2-\tfrac\ep {2M}\sfd(x,C)\big)\wedge 1
$).

Then the function $\chi f$ still satisfies \eqref{eq:lipaug}, has bounded support and Lipschitz constant bounded by
\[
\Lip(\chi f)\leq \Lip(\chi)\sup|f|+\Lip(f)\sup|\chi|\leq \frac\ep{2M}M+L+\frac\ep2=L+\ep,
\]
thus the proof is completed.
\end{proof}
\section{Application to Sobolev Spaces}

In this section we discuss the  application of our extension result to the study Sobolev spaces that we presented in the introduction.

We shall denote by ${\rm Lip}_{bs} (\X)$ the space of Lipschitz functions  on $\X$ with bounded support. 
\begin{defn}[Sobolev Spaces] Let $(\X,\sfd)$ be a complete metric space and $\mm$ a non-negative and non-zero Borel measure on $\X$. For $p\geq 1$ we define the functional $\Ch_{p,\X} : L^p(\X, \mm) \to [0 , \infty]$ as
\[
\Ch_{p,\X} ( f) : = \inf \left\{ \liminf_{n \to \infty} \int_\X \lipa{f_n}{x}^p \, \d\mm (x) \; : \;  f_n \in {\rm Lip}_{bs}(\X) ,\ f_n \to f \text{ in }L^p(\X, \mm) \right\}.
\]
\end{defn}
The functional $\Ch_{p,\X} $ is central in the definition of the BV/Sobolev spaces on metric measure spaces, as for $p>1$ one puts $W^{1,p}(\X,\sfd, \mm):=\{\Ch_{p,\X}(f) < \infty\}$, while for $p=1$ the domain of finiteness of $\Ch_{1,\X}$ is defined to be the space of functions of bounded variations.

Notice that if $\Y\subset\X$ is a set where $\mm$ is concentrated, then we can naturally identify the spaces $L^p(\X,\mm)$ and $L^p(\Y,\mm)$: we will constantly do this in the next result:
\begin{thm}\label{thm:sobapp} Let $(\X,\sfd)$ be a complete metric space and $\mm$ a non-negative and non-zero Borel measure on $\X$. Let  $\Y\subset\X$ be a closed subset on which $\mm$ is concentrated and, for brevity, denote by $\sfd,\mm$ the restrictions to $\Y$ of the distance and measure on $\X$, respectively.

Then for every $p\geq1$ the functionals  $\Ch_{p,\X}$ and $\Ch_{p,\Y}$ coincide on $L^p(\X, \mm)\sim L^p(\Y, \mm)$.
\end{thm}
\begin{proof} Let $f \in L^p(\X,\mm)$ and $(f_n)\subset {\rm Lip}_{bs}(\X)$ a sequence that converges to $f$ in the $L^p(\X,\mm)$-norm. Put $g_n:=f_n\restr \Y\in {\rm Lip}_{bs}(\Y)$ and notice that evidently    $g_n \to f$ in $L^p(\Y,\mm)$ and moreover $\lipa{g_n}{x} \leq \lipa{f_n}{x} $ for every $x \in \Y$. In particular we have
\[
 \Ch_{p,\Y} (f)  \leq \liminf_{n \to \infty} \int_\Y \lipa{g_n}{x}^p \, \d \mm(x) \leq \liminf_{n \to \infty} \int_\X \lipa{f_n}{x}^p \, \d \mm(x)
 \]
and  taking the infimum with respect to all the possible sequences $f_n$ we get $ \Ch_{p,\Y}(f) \leq  \Ch_{p,\X}(f)$. 

To prove the opposite inequality we shall make use of our Theorem \ref{thm:main}. Let   $(g_n)\subset {\rm Lip}_{bs}(\Y)$ be converging to $f$ in  $L^p(\Y,\mm)$ and use  Theorem \ref{thm:main} to obtain the existence of functions $(f_n)\subset {\rm Lip}_{bs}(\X)$ coinciding with $g_n$ on $\Y$ and such that $\lipa{f_n}{x} = \lipa{g_n}{x}$ for every $x\in\Y$. In particular, we still have $f_n\to f$ in $L^p(\X,\mm)$ and therefore
\[
 \Ch_{p,\X} (f)  \leq \liminf_{n \to \infty} \int_\X \lipa{f_n}{x}^p \, \d \mm(x) = \liminf_{n \to \infty} \int_\Y \lipa{g_n}{x}^p \, \d \mm(x),
 \]
so that, again, taking the infimum with respect to all approximating sequences $(g_n)$ yields  $\Ch_{p,\X}(f) \leq \Ch_{p,\Y}(f)$, concluding  the proof.
\end{proof}

\def\cprime{$'$} \def\cprime{$'$}

\end{document}